\documentclass[11pt,a4paper]{article}

\usepackage[utf8]{inputenc}
\usepackage[T1]{fontenc}
\usepackage[english]{babel} 
\usepackage{amsmath}
\usepackage{amsfonts}
\usepackage{amsthm}
\usepackage{amssymb}
\usepackage{makeidx}
\usepackage{graphicx}
\usepackage{lmodern}
\usepackage[left=2cm,right=2cm,top=2cm,bottom=2cm]{geometry}

\usepackage{color}
\usepackage{comment}
\usepackage[dvipsnames]{xcolor}
\usepackage{graphicx}
\usepackage{framed}
\usepackage{tikz}
\usepackage{calrsfs}
\DeclareMathAlphabet{\pazocal}{OMS}{zplm}{m}{n}
\usepackage{bm}

\usepackage{fourier-orns}
\usepackage{mathtools}
\usepackage{relsize}
\usepackage{dsfont}
\usepackage{comment}

\usepackage{hyperref} 
\hypersetup{
	linktocpage = true,
	colorlinks = true,
	linkcolor = {RoyalBlue},
	urlcolor = {RoyalBlue},
	citecolor = {Green}}

\newcommand{\B}{\mathbb{B}}
\newcommand{\R}{\mathbb{R}}

\newcommand{\Fpazo}{\pazocal{F}}

\newcommand{\Upazo}{\pazocal{U}}

\newcommand{\Mpazo}{\pazocal{M}}

\newcommand{\Qpazo}{\pazocal{Q}}

\newcommand{\Rpazo}{\pazocal{R}}
\newcommand{\Spazo}{\pazocal{S}}
\newcommand{\Opazo}{\pazocal{O}}
\newcommand{\Wpazo}{\pazocal{W}}

\newcommand{\Dcal}{\mathcal{D}}

\newcommand{\Lcal}{\mathcal{L}}
\newcommand{\Vcal}{\mathcal{V}}
\newcommand{\Pcal}{\mathcal{P}}

\newcommand{\Scal}{\mathcal{S}}

\newcommand{\Tcal}{\mathcal{T}}

\newcommand{\Id}{\textnormal{Id}}

\newcommand{\D}{\textnormal{D}}

\newcommand{\supp}{\textnormal{supp}}

\newcommand{\Lip}{\textnormal{Lip}}
\newcommand{\AC}{\textnormal{AC}}

\newcommand{\Div}{\textnormal{div}}

\newcommand{\dist}{\textnormal{dist}}
\newcommand{\textbn}[1]{\textnormal{\textbf{#1}}}
\newcommand{\co}{\overline{\textnormal{co}} \hspace{0.05cm}}

\newcommand{\Bpi}{\boldsymbol{\pi}}

\newcommand{\Bmu}{\boldsymbol{\mu}}

\newcommand{\INTDom}[3]{\int_{#2} #1 \textnormal{d} #3}
\newcommand{\INTSeg}[4]{\int_{#3}^{#4} #1 \textnormal{d} #2}
\newcommand{\NormL}[3]{\parallel \hspace{-0.1cm} #1 \hspace{-0.1cm} \parallel _ {L^{#2}(#3)}}
\newcommand{\NormC}[3]{\left\| #1  \right\| _ {C^{#2}(#3)}}
\newcommand{\Norm}[1]{\parallel \hspace{-0.1cm} #1 \hspace{-0.1cm} \parallel}

\newtheorem{Def}{Definition}[section]
\newtheorem{thm}[Def]{Theorem}
\newtheorem{prop}[Def]{Proposition}

\numberwithin{equation}{section}
\renewcommand{\epsilon}{\varepsilon}

\newtheorem{taggedhypx}{Hypotheses}
\newenvironment{taggedhyp}[1]
    {\renewcommand\thetaggedhypx{#1}\taggedhypx}
    {\endtaggedhypx}

\title{Viability and Exponentially Stable Trajectories for Differential Inclusions in Wasserstein Spaces}

\author{Benoît Bonnet\footnote{CNRS, LAAS, 7 avenue du colonel Roche, F-31400 Toulouse, France. \textit{E-mail}: \texttt{benoit.bonnet@laas.fr} (Corresponding author)} \,and Hélène Frankowska\footnote{CNRS,  IMJ-PRG,  UMR  7586,  Sorbonne  Université,  4  place  Jussieu,  75252  Paris,  France. \hfill \hspace{3.5cm} \textit{E-mail}: \texttt{helene.frankowska@imj-prg.fr}}}

\begin{document}

\maketitle

\begin{abstract}
In this article, we prove a general viability theorem for continuity inclusions in Wasserstein spaces, and provide an application thereof to the existence of exponentially stable trajectories obtained via the second method of Lyapunov.
\end{abstract}


\section{Introduction}

During the past decade, the study of continuity equations in the space of measures has gained a tremendous amount of steam. Originally motivated by applications to crowd motion \cite{CPT,FornasierPR2014,Pedestrian}, opinion propagation \cite{AlbiPZ2017,ControlKCS} and game theory \cite{CDLL,Huang2006,Lasry2007}, the investigation of the mathematical properties of multi-agent systems -- mostly studied via optimal transport techniques in the mean-field setting -- has become a broad field of research at the intersection between pure, applied and computational mathematics. In this context, the literature devoted to the analysis of control problems formulated in the so-called \textit{Wasserstein spaces} has been steadily growing for several years (see e.g. \cite{Badreddine2022,ContInc,SetValuedPMP,SemiSensitivity,PMPWass,Burger2021,Cavagnari2021,Fornasier2019} and the references therein).

The aim of this paper is to present a novel viability result for set-valued dynamics in Wasserstein spaces, following the terminology introduced in our previous work \cite{ContInc}. The concept of viability, which goes back to the eighties for differential inclusions, has tremendous applications in control theory, e.g. to ensure the existence of solutions to state-constrained control systems \cite[Chapter 10]{Aubin1990}, to derive optimality conditions in the form of Hamilton-Jacobi-Bellman equations \cite{Frankowska1995} or to characterise the existence of Lyapunov stable trajectories \cite{Frankowska1996}. Motivated by these aspects, we derive a general viability result for set-valued dynamics in Wasserstein spaces in Theorem \ref{thm:Viability}. The latter relies strongly on Theorem \ref{thm:Velocity}, which is a technical prerequisite ensuring the existence of solutions with prescribed initial velocities to continuity inclusions. Finally, we apply these results in Theorem \ref{thm:Lyapunov}, where we illustrate how they can be used to obtain exponentially stable trajectories in terms of a given Lyapunov function.

The organisation of the paper is the following. After recollecting preliminary notions of optimal transport and functional analysis in Section \ref{section:Preli}, we prove the existence of solutions to continuity inclusions with prescribed initial velocities in Section \ref{section:Velocity}, and use the corresponding result to prove the viability theorem in Section \ref{section:Viability}. Then, in Section \ref{section:Lyapunov}, we discuss the existence of exponentially stable trajectories via the second method of Lyapunov.


\section{Preliminaries}
\label{section:Preli}

In this section, we recall preliminary notions of measure theory, optimal transport and set-valued analysis, for which we refer to \cite{AmbrosioFuscoPallara,AGS} and \cite{Aubin1990} respectively.


\subsection{Optimal transport and calculus in Wasserstein spaces}

In the sequel, $\Pcal(\Omega)$ will denote the space of Borel probability measures over a Borel measurable set $\Omega \subset \R^d$, endowed with the standard \textit{narrow topology} (see \cite[Chapter 5]{AGS}). Let $\Pcal_2(\R^d)$ be the subset of probability measures whose $2$-momentum $\Mpazo_2^2(\mu) := \INTDom{|x|^2}{\R^d}{\mu(x)}$ is finite, and $\Pcal_c(\R^d)$ be that of measures with compact support. Given $\mu \in \Pcal(\R^d)$ and $p \in [1,+\infty)$, we denote by $(L^p(\R^d,\R^d;\mu),\Norm{\cdot}_{L^p(\mu)})$ the Banach space of maps from $\R^d$ into itself that are $p$-summable with respect to $\mu$, and by $L^{\infty}(\R^d,\R^d;\mu)$ that of $\mu$-essentially bounded maps. We will also denote by $\Lcal^1$ the standard Lebesgue measure over $\R$, and use the shorter notation $(L^p(I,\R_+),\Norm{\cdot}_p)$ with $p \in [1,+\infty]$ for the Lebesgue spaces of maps going from an interval $I \subset \R$ into $\R_+$.

Given a Borel map $f : \R^d \rightarrow \R^d$, we define the \textit{pushforward} $f_{\sharp} \mu \in \Pcal(\R^d)$ of $\mu$ through $f$ as the unique measure satisfying
$ f_{\sharp} \mu(B) := \mu(f^{-1}(B))$ for every Borel set $B \subset \R^d$. Using this notation, we can define the set of \textit{transport plans} between two elements $\mu,\nu \in \Pcal(\R^d)$ as
\begin{equation*}
\Gamma(\mu,\nu) := \Big\{ \gamma \in \Pcal(\R^{2d}) ~\, \textnormal{s.t.}~ \pi^1_{\sharp} \gamma = \mu ~~\text{and}~~ \pi^2_{\sharp} \gamma = \nu \Big\},
\end{equation*}
where $\pi^1,\pi^2 : \R^d \times \R^d \to \R^d$ are the projections onto the first and second factors respectively. Leveraging this notion, we can in turn recall the notion of \textit{Wasserstein distance} between measures.

\begin{Def}[Wasserstein distance]
\label{def:Wass}
The quantity
\begin{equation*}
W_2(\mu,\nu) := \min_{\gamma \in \Gamma(\mu,\nu)} \bigg( \INTDom{|x-y|^2}{\R^{2d}}{\gamma(x,y)} \bigg)^{1/2},
\end{equation*}
defines a distance between any two measures $\mu,\nu \in \Pcal_2(\R^d)$, and we denote by $\Gamma_o(\mu,\nu)$ the (nonempty) set of transport plans at which the minimum is attained.
\end{Def}

Following \cite{AGS,Otto2001}, the complete separable metric space $(\Pcal_2(\R^d),W_2)$ -- usually called \textit{Wasserstein space} -- can be formally endowed with the structure of a differentiable manifold. In the sequel, we will also consider the (non-complete) metric space $(\Pcal_c(\R^d),W_2)$ of compactly supported measures equipped with the $W_2$-metric. We end this first preliminary section by stating a simplified version of a pivotal result of Wasserstein calculus, allowing to describe the superdifferential of the squared Wasserstein distance (see e.g. \cite[Theorem 10.2.2]{AGS}).

\begin{prop}
\label{prop:WassSuperdiff}
For every $\mu,\nu \in \Pcal_2(\R^d)$ and each $\gamma \in \Gamma_o(\mu,\nu)$, it holds that
\begin{equation*}
\begin{aligned}
\tfrac{1}{2} W^2_2 \big( (\Id + & h \xi)_{\sharp} \mu, \nu \big) - \tfrac{1}{2}W_2^2(\mu,\nu) \\
& \leq h \INTDom{\langle \xi(x) , x-y \rangle}{\R^{2d}}{\gamma(x,y)} + h^2 \Norm{\xi}_{L^2(\mu)}^2
\end{aligned}
\end{equation*}
for every $\xi \in L^2(\R^d,\R^d;\mu)$ and each $h > 0$.
\end{prop}


\subsection{Elements of set-valued analysis}

In what follows, given a complete separable metric space $(\Scal,d_{\Scal}(\cdot,\cdot))$ and a Fr\'echet space $(E,d_E(\cdot,\cdot))$ (see e.g. \cite{Horvath2012}), we shall write $\Fpazo : \Scal\rightrightarrows E$ to denote \textit{set-valued maps} from $\Scal$ into $E$. We will also denote by $C^0(\R^d,\Scal)$ and $\AC([0,T],\Scal)$ the spaces of continuous and absolutely continuous maps from $\R^d$ and $[0,T]$ into $\Scal$ respectively, and write $\Lip(\phi \, ; K)$ for the Lipschitz constant of a map $\phi : \R^d \to \Scal$ over some set $K \subset \R^d$.

In the coming definitions, we recall the notions of \textit{measurability} and \textit{lower-semicontinuity} for multifunctions. Therein and in general, $\B_{\Scal}(s,r)$ will stand for the ball of radius $r>0$ centered at $s \in \Scal$.

\begin{Def}[Measurability]
A set-valued mapping $\Fpazo : [0,T] \rightrightarrows E$ is $\Lcal^1$-measurable provided that
\begin{equation*}
\Fpazo^{-1}(\Opazo) := \Big\{ s \in \Scal ~\, \text{s.t.}~ \Fpazo(s) \cap \Opazo \neq \emptyset \Big\}
\end{equation*}
is $\Lcal^1$-measurable for each open set $\Opazo \subset E$. We then say that an $\Lcal^1$-measurable function $t \in [0,T] \mapsto f(t) \in \Fpazo(t)$ is a \textit{measurable selection}.
\end{Def}

\begin{Def}[Lower-semicontinuity]
A set-valued mapping $\Fpazo : \Scal \rightrightarrows E$ is lower-semicontinuous at $s \in \Scal$ if for each open set $\Upazo \subset E$ such that $\Fpazo(s) \cap \Upazo \neq \emptyset$, there exists $\delta >0$ such that $\Fpazo(s') \cap \Upazo \neq \emptyset$ for all $s' \in \B_{\Scal}(s,\delta)$.
\end{Def}

In what follows, we will denote by $\co B$ the \textit{closed convex hull} of a set $B \subset E$, defined by
\begin{equation*}
\co B := \overline{\bigcup_{N \geq 1}\bigg\{ \mathsmaller{\sum}\limits_{i=1}^N \alpha_i b_i \,~\text{s.t.}~ \alpha_i \geq 0, \; b_i \in B, \; \mathsmaller{\sum}\limits_{i=1}^N \alpha_i = 1 \bigg\}}^E.
\end{equation*}
We will also use the standard notations $\textnormal{int}(\Qpazo)$ and $\partial \Qpazo := \Qpazo \setminus \textnormal{int}(\Qpazo)$ to refer to the interior and topological boundary of a closed set $\Qpazo \subset \Scal$ respectively, as well as
\begin{equation*}
\dist_{\Scal} (\Qpazo \, ; \Rpazo) := \inf \Big\{ d_{\Scal}(s,s') ~\, \textnormal{s.t.}~ s \in \Qpazo ~\text{and}~ s' \in \Rpazo \Big\}
\end{equation*}
for the distance between two closed subsets of $\Scal$. We finally recall that $(C^0(K,\R^d),\NormC{\cdot}{0}{K,\R^d})$ is a separable Banach space whenever $K \subset \R^d$ is compact.


\subsection{Continuity equations and inclusions in $(\Pcal_c(\R^d),W_2)$}

In the sequel, we will consider measure dynamics described by \textit{continuity equations} of the form
\begin{equation}
\label{eq:CE}
\partial_t \mu(t) + \Div_x(v(t)\mu(t)) = 0,
\end{equation}
whose solutions are understood in the sense of distributions, namely
\begin{equation*}
\INTSeg{\INTDom{\Big( \partial_t \phi(t,x) + \big\langle \nabla_x \phi(t,x) , v(t,x) \big\rangle \Big)}{\R^d}{\mu(t)(x)}}{t}{0}{T} = 0,
\end{equation*}
for each $\phi \in C^{\infty}_c((0,T) \times \R^d)$. Here and in what follows, we shall assume that the \textit{velocity-fields} $v : [0,T] \times \R^d \rightarrow \R^d$ satisfy the following assumptions.

\begin{taggedhyp}{\textbn{(CE)}} \hfill
\label{hyp:CE}
\begin{enumerate}
\item[$(i)$] The application $t \in [0,T] \mapsto v(t,x) \in \R^d$ is $\Lcal^1$-measurable for all $x \in \R^d$, and there exists a map $m(\cdot) \in L^1([0,T],\R_+)$ such that
\begin{equation*}
|v(t,x)| \leq m(t) \big( 1+|x| \big),
\end{equation*}
for $\Lcal^1$-almost every $t \in [0,T]$ and all $x \in \R^d$.
\item[$(ii)$] For each compact set $K \subset \R^d$, there exists a map $l_K(\cdot) \in L^1([0,T],\R_+)$ such that
\begin{equation*}
\Lip(v(t) \, ;K) \leq l_K(t),
\end{equation*}
for $\Lcal^1$-almost every $t \in [0,T]$.
\end{enumerate}
\end{taggedhyp}

Under these Cauchy-Lipschitz assumptions, we have the following well-posedness result for \eqref{eq:CE} (see e.g. \cite{ContInc}).

\begin{thm}[Well-posedness of \eqref{eq:CE}]
\label{thm:Wellposedness}
Let $r >0$ and suppose that $v : [0,T] \times \R^d \to \R^d$ satisfies hypotheses \ref{hyp:CE}. Then for each $(\tau,\mu_{\tau}) \in [0,T] \times \Pcal(B(0,r))$, there exists a unique solution $\mu(\cdot) \in \AC([\tau,T],\Pcal_c(\R^d))$ to
\begin{equation}
\label{eq:CEtau}
\left\{
\begin{aligned}
& \partial_t \mu(t) + \Div_x(v(t)\mu(t)) = 0, \\
& \mu(\tau) = \mu_{\tau}.
\end{aligned}
\right.
\end{equation}
Moreover, the curve $\mu(\cdot)$ satisfies
\begin{equation}
\label{eq:ACEst}
\supp(\mu(t)) \subset B(0,R_r), \hspace{0.3cm} W_2(\mu(t),\mu(s)) \leq c_r \INTSeg{m(\zeta)}{\zeta}{s}{t},
\end{equation}
for all $\tau \leq s \leq t \leq T$, where $R_r,c_r >0$ depend only on the magnitudes of $r,\Norm{m(\cdot)}_1$. Furthermore, the latter can be represented explicitly as
\begin{equation}
\label{eq:Semigroup}
\mu(t) = \Phi_{(\tau,t)}^v(\cdot)_{\sharp} \mu(\tau),
\end{equation}
for every $t \in [\tau,T]$, where $(\Phi_{(\tau,t)}(\cdot))_{t \in [0,T]}$ are the \textit{flows of diffeomorphisms} defined as the unique solution of
\begin{equation*}
\left\{
\begin{aligned}
\partial_t \Phi_{(\tau,t)}^v(x) & = v \big( t ,\Phi_{(\tau,t)}^v(x)\big), \\
\Phi_{(\tau,\tau)}^v(x) & = x,
\end{aligned}
\right.
\end{equation*}
for all $x \in \R^d$.
\end{thm}

In our previous work \cite{ContInc}, we proposed a set-valued generalisation of Cauchy problems of the form \eqref{eq:CE}, for which the right-hand sides are \textit{set-valued maps}
\begin{equation*}
V : [0,T] \times \Pcal_2(\R^d) \rightrightarrows C^0(\R^d,\R^d),
\end{equation*}
whose images typically lie within subsets of locally Lipschitz and sublinear vector fields.

\begin{Def}[Continuity inclusions]
\label{def:Inc}
A curve of measures $\mu(\cdot)  \in \AC([\tau,T],\Pcal_c(\R^d))$ is a solution of
\begin{equation}
\label{eq:CI}
\left\{
\begin{aligned}
& \partial_t \mu(t) \in - \Div_x \Big( V(t,\mu(t)) \mu(t) \Big), \\
& \mu(\tau) = \mu_{\tau},
\end{aligned}
\right.
\end{equation}
if there exists a map $t \in [0,T] \mapsto v(t) \in V(t,\mu(t))$ such that $\mu(\cdot)$ solves \eqref{eq:CEtau}, with $t \in [0,T] \mapsto v(t)_{|K} \in C^0(K,\R^d)$ being $\Lcal^1$-measurable for each compact set $K \subset \R^d$.
\end{Def}

Throughout this article, we will impose the following set of assumptions on the set-valued map $V(\cdot,\cdot)$, and will sometimes use the notation
\begin{equation*}
V(t,\mu)_{|K} := \Big\{ v_{|K} \in C^0(K,\R^d) ~\, \text{s.t.}~ v \in V(t,\mu) \Big\}
\end{equation*}
where $K \subset \R^d$ is a compact set.

\begin{taggedhyp}{\textbn{(CI)}} \hfill
\label{hyp:CI}
\begin{enumerate}
\item[$(i)$] The set-valued map $t \in [0,T] \rightrightarrows V(t,\mu)_{|K} \subset C^0(K,\R^d)$ is lower-semicontinuous with closed non-empty images for all $\mu \in \Pcal(K)$ whenever $K \subset \R^d$ is a compact set.
\item[$(ii)$] There exists a map $m(\cdot) \in L^1([0,T],\R_+)$ such that for all $\mu \in \Pcal_2(\R^d)$ and each $v \in V(t,\mu)$, it holds
\begin{equation*}
|v(x)| \leq m(t) \Big( 1+ |x| + \Mpazo_2(\mu) \Big),
\end{equation*}
for $\Lcal^1$-almost every $t \in [0,T]$ and all $x \in \R^d$.
\item[$(iii)$] There exists $l(\cdot) \in L^1([0,T],\R_+)$ such that for $\Lcal^1$-almost every $t \in [0,T]$, all $\mu \in \Pcal_2(\R^d)$ and each $v \in V(t,\mu)$, it holds
\begin{equation*}
\Lip(v \, ; \R^d) \leq l(t).
\end{equation*}
\item[$(iv)$] There exists $L(\cdot) \in L^1([0,T],\R_+)$ such that for $\Lcal^1$-almost every $t \in [0,T]$, all $\mu,\nu \in \Pcal_2(\R^d)$ and each $v \in V(t,\mu)$, there exists $w \in V(t,\nu)$ such that
\begin{equation*}
\begin{aligned}
\sup_{x \in \R^d} |v(x) - w(x)| \leq L(t) W_2(\mu,\nu).
\end{aligned}
\end{equation*}
\end{enumerate}
\end{taggedhyp}

\medskip

We would like to stress that hypothesis \ref{hyp:CI}-$(i)$ is not sharp compared to its natural measurability counterparts in \ref{hyp:CE}-$(i)$ or \cite[p.608]{ContInc}, but greatly simplifies the proofs of Section \ref{section:Velocity}. Similarly, one could opt for localised versions of \ref{hyp:CI}-$(iii)$ and $(iv)$, at the price of extra technicalities. For the sake of simplicity and readability, we defer the investigation of viability properties in such a general context to a subsequent article.

In the following theorem, we recall a condensed version of the Filippov estimates derived in \cite[Theorem 4]{ContInc}.

\begin{thm}[Filippov estimates]
\label{thm:Filippov}
Let $w : [0,T] \times \R^d \to \R^d$ be a velocity field satisfying hypotheses \ref{hyp:CE}, $\nu^0 \in \Pcal_c(\R^d)$ and $\nu(\cdot)$ be the unique solution of
\begin{equation*}
\partial_t \nu(t) + \Div_x (w(t)\nu(t)) = 0, \qquad \nu(0) = \nu^0.
\end{equation*}
Moreover, let  $r > 0$ be such that $\supp(\nu(t)) \subset B(0,r)$ for all $t \in [0,T]$, and consider the Lebesgue integrable map
\begin{equation*}
\eta_{\nu} : t \in [0,T] \mapsto \dist_{C^0(B(0,r),\R^d)} \Big( w(t) \, ; V(t,\nu(t))_{|B(0,r)} \Big).
\end{equation*}
Then for every $(\tau,\mu_{\tau}) \in [0,T] \times \Pcal(B(0,r))$, there exists a solution $\mu(\cdot) \in \AC([\tau,T],\Pcal_c(\R^d))$ to \eqref{eq:CI} such that $\supp(\mu(t)) \subset B(0,R_r)$ and
\begin{equation*}
W_2(\mu(t),\nu(t)) \leq C_r \bigg( W_2(\mu_{\tau},\nu(\tau)) + \INTSeg{\eta_{\nu}(s)}{s}{\tau}{t} \bigg)
\end{equation*}
for all times $t \in [\tau,T]$. Therein, $R_r > 0$ only depends on the magnitudes of $r,\Norm{m(\cdot)}_1$, while $C_r > 0$ only depends on those of $r,\Norm{m(\cdot)}_1,\Norm{l(\cdot)}_1$ and $\Norm{L(\cdot)}_1$.
\end{thm}

In what follows, we will often work with the \textit{reachable} and \textit{solution} sets of the Cauchy problem \eqref{eq:CI}.

\begin{Def}[Reachable and solution sets]
Given $\mu^0 \in \Pcal_c(\R^d)$, we define the solution set of \eqref{eq:CI} as
\begin{equation*}
\begin{aligned}
\Spazo_{[0,T]}(\mu^0) \hspace{-0.075cm} := & \Big\{ \mu(\cdot) \in \AC([0,T],\Pcal_c(\R^d)) ~\text{solution of \eqref{eq:CI}} \\
& \hspace{3.25cm} \text{with $(\tau,\mu_{\tau}) = (0,\mu^0)$} \Big\}
\end{aligned}
\end{equation*}
and similarly, we denote the underlying reachable set at time $t \in [0,T]$ by
\begin{equation}
\label{eq:Reachable}
\Rpazo_t(\mu^0) := \Big\{ \mu(t) ~\,\text{such that $\mu(\cdot) \in \Spazo_{[0,T]}(\mu^0)$} \Big\}.
\end{equation}
\end{Def}

\smallskip

\begin{thm}[Properties of $\Spazo_{[0,T]}(\mu^0)$ and $\Rpazo_t(\mu^0)$]
\label{thm:Compactness}
Let $V : [0,T] \times \Pcal_2(\R^d) \rightrightarrows C^0(\R^d,\R^d)$ be a set-valued map with convex images satisfying hypotheses \ref{hyp:CI}.

Then for each $r > 0$ and any $\mu^0 \in \Pcal(B(0,r))$, the sets $\Spazo_{[0,T]}(\mu^0) \subset C^0([0,T],\Pcal_2(B(0,R_r)))$ and $\Rpazo_t(\mu^0) \subset \Pcal_2(B(0,R_r))$ are compact, with $R_r > 0$ being as in Theorem \ref{thm:Filippov}. Moreover, the reachable sets satisfy the semigroup property
\begin{equation}
\label{eq:SemigroupReach}
\Rpazo_t(\mu^0) = \Rpazo_{t-\tau}(\Rpazo_{\tau}(\mu^0)),
\end{equation}
for all times $0 \leq \tau \leq t \leq T$.
\end{thm}

\begin{proof}
See the arguments in \cite[Theorem 6]{ContInc}.
\end{proof}


\section{Existence of measure curves with prescribed initial velocities}
\label{section:Velocity}

In this section, we establish the existence of solutions to \eqref{eq:CI} with prescribed initial velocities. This fairly non-trivial result will be instrumental in the proof of the viability theorem of Section \ref{section:Viability}.

\begin{thm}[Curves with given initial velocities]
\label{thm:Velocity}
Let $V : [0,T] \times \Pcal_2(\R^d) \rightrightarrows C^0(\R^d,\R^d)$ be a set-value map satisfying hypotheses \ref{hyp:CI}.

Then, there exists a set $\Tcal \subset (0,T)$ of full $\Lcal^1$-measure such that for each $r> 0$, any $(\tau,\mu_{\tau}) \in \Tcal \times \Pcal(B(0,r))$ and $v_{\tau} \in V(\tau,\mu_{\tau})$, there exists a solution $\mu(\cdot) \in \AC([\tau,T],\Pcal_c(\R^d))$ of \eqref{eq:CI} such that
\begin{equation}
\label{eq:ThmVelExp}
W_2 \big(\mu(\tau+h) , (\Id + h v_{\tau})_{\sharp} \mu_{\tau} \big) = o_{\tau}(h),
\end{equation}
for all $h >0$ sufficiently small.
\end{thm}

\begin{proof}
First and foremost, let it be noted that by applying Theorem \ref{thm:Filippov} with a constant curve of measures $\nu(\cdot) \equiv \nu \in \Pcal(B(0,r))$, there exists $R_r > 0$ such that
\begin{equation*}
\supp(\mu(t)) \subset B(0,R_r),
\end{equation*}
for any solution of \eqref{eq:CI} starting from $\mu_{\tau} \in \Pcal(B(0,r))$ at time $\tau \in [0,T]$. We also define $\Tcal \subset (0,T)$ of full $\Lcal^1$-measure as the intersection of the sets of Lebesgue points (see e.g. \cite[Corollary 2.23]{AmbrosioFuscoPallara}) of $m(\cdot),l(\cdot)$ and $L(\cdot)$ at which hypotheses \ref{hyp:CI}-$(ii)$, $(iii)$ and $(iv)$ hold.

Fix $\tau \in \Tcal$ and observe that, as a consequence of hypotheses \ref{hyp:CI}, every $v_{\tau} \in V(\tau,\mu_{\tau})$ satisfies hypotheses \ref{hyp:CE}. Whence, there exists a unique solution $\nu(\cdot) \in \AC([\tau,T],\Pcal_c(\R^d))$ to the Cauchy problem
\begin{equation*}
\left\{
\begin{aligned}
& \partial_t \nu(t) + \Div_x (v_{\tau} \nu(t)) = 0, \\
& \nu(\tau) = \mu_{\tau},
\end{aligned}
\right.
\end{equation*}
which can be represented as
\begin{equation*}
\nu(t) = \Phi^{v_{\tau}}_{(\tau,t)}(\cdot)_{\sharp} \mu_{\tau}.
\end{equation*}
Under hypotheses \ref{hyp:CE}, it follows from standard linearisation techniques for characteristic flows (see e.g. \cite[Appendix A]{SemiSensitivity}) that
\begin{equation}
\label{eq:Smallo1}
\Phi_{(\tau,\tau+h)}^{v_{\tau}}(x) = x + h v_{\tau}(x) + o_{\tau,x}(h),
\end{equation}
where $\INTSeg{\sup_{x \in B(0,R)}|o_{\tau,x}(h)|}{\tau}{0}{T} = o_{R}(h)$ for each $R>0$. Furthermore, upon noticing that
\begin{equation*}
\Big( \Id + h v_{\tau} , \Phi_{(\tau,\tau+h)}^{v_{\tau}} \Big)_{\raisebox{4pt}{$\scriptstyle{\sharp}$}} \mu_{\tau} \in \Gamma \Big( (\Id + h v_{\tau})_{\sharp} \mu_{\tau},\nu(\tau+h) \Big)
\end{equation*}
we can easily get from \eqref{eq:Smallo1} the following distance estimate
\begin{equation}
\label{eq:Smallo1bis}
\begin{aligned}
W_2 \big( \nu(\tau+h),&(\Id + hv_{\tau})_{\sharp} \mu_{\tau}  \big) \\
& \leq \, \NormL{\Phi_{(\tau,\tau+h)}^{v_{\tau}} - \Id - h v_{\tau}}{2}{\mu_{\tau}} = o_{\tau}(h).
\end{aligned}
\end{equation}
Setting $K := B(0,R_r)$, observe that under hypothesis \ref{hyp:CI}-$(i)$, the set-valued map $t \in [0,T] \rightrightarrows V(t,\mu_{\tau})_{|K}$ is lower-semicontinuous. Thus, for each $\varepsilon > 0$, there exists $\delta > 0$ such that
\begin{equation*}
v_{\tau} \in V(t,\mu_{\tau})_{|K} \, + \varepsilon \B_{C^0(K,\R^d)},
\end{equation*}
for all times $t \in [\tau,\tau+\delta]$. In turn, by using hypothesis \ref{hyp:CI}-$(iv)$ in the previous identity, it further holds that
\begin{equation}
\label{eq:vtauEst1}
v_{\tau} \in V(t,\nu(t))_{|K} \, + \Big( \varepsilon + L(t) W_2(\mu_{\tau},\nu(t)) \Big) \B_{C^0(K,\R^d)}.
\end{equation}
Noticing that by construction, the curve $\nu(\cdot)$ satisfies
\begin{equation*}
W_2(\mu_{\tau},\nu(t)) \leq c_r m(\tau) (t-\tau)
\end{equation*}
for all times $t \in [\tau,T]$, it then follows from \eqref{eq:vtauEst1} that
\begin{equation}
\label{eq:vtauEst2}
\begin{aligned}
\INTSeg{\dist_{C^0(K,\R^d)} \Big(&  v_{\tau} \, ; V(t,\nu(t))_{|K} \Big)}{t}{\tau}{\tau+h} \\
& \leq \varepsilon h + c_r m(\tau) \, h \INTSeg{L(t)}{t}{\tau}{\tau+h}
\end{aligned}
\end{equation}
for each $h > 0$ sufficiently small. By applying Theorem \ref{thm:Filippov} in conjunction with \eqref{eq:vtauEst2} while recalling that $\tau \in \Tcal$ is a Lebesgue point of $L(\cdot)$, we obtain the existence of a solution $\mu(\cdot) \in \AC([\tau,T],\Pcal_c(\R^d))$ of \eqref{eq:CI} such that
\begin{equation}
\label{eq:Smallo2}
W_2(\mu(\tau+h),\nu(\tau+h)) \leq C_r \varepsilon h + o_{\tau}(h),
\end{equation}
for every $\varepsilon > 0$, whenever $h>0$ is sufficiently small. Whence, by merging the estimates of \eqref{eq:Smallo1bis} and \eqref{eq:Smallo2}, we can finally conclude that the curve $\mu(\cdot) \in \AC([\tau,T],\Pcal_c(\R^d))$ satisfies \eqref{eq:ThmVelExp}.
\end{proof}


\section{Viability for continuity inclusions}
\label{section:Viability}

In this section, we prove a general viability result for \eqref{eq:CI} involving the \textit{contingent cone} to the set of constraints.

\begin{Def}[Contingent cones]
The contingent cone to a set $\Qpazo \subset \Pcal_2(\R^d)$ at some $\mu \in \Qpazo$ is defined by
\begin{equation*}
\label{eq:AdjacentCone}
\begin{aligned}
T_{\Qpazo}(\mu) := \bigg\{ & \xi \in L^2(\R^d,\R^d;\mu) ~\, \textnormal{s.t. there exists $h_i \to 0^+$} \\
& \, \text{for which}~ \dist_{\Pcal_2} \big( (\Id + h_i \xi)_{\sharp} \mu \, ; \Qpazo \big) = o(h_i) \bigg\}.
\end{aligned}
\end{equation*}
\end{Def}

In what follows, we shall say that a subset $\Qpazo \subset \Pcal_2(\R^d)$ is \textit{proper} if $\Qpazo  \cap \B_{\Pcal_2}(\mu,r)$ is compact for every $\mu \in \Qpazo$ and each $r > 0$.

\begin{thm}[Viability for proper constraints]
\label{thm:Viability}
Let $V : [0,T] \times \Pcal_2(\R^d) \rightrightarrows C^0(\R^d,\R^d)$ be a set-valued map with convex images satisfying hypotheses \ref{hyp:CI}, and $\Qpazo \subset \Pcal_2(\R^d)$ be a proper set such that
\begin{equation*}
V(t,\nu) \cap \co T_{\Qpazo}(\nu) \neq \emptyset
\end{equation*}
for $\Lcal^1$-almost every $t \in [0,T]$ and each $\nu \in \Qpazo$. Then for each $\mu^0 \in \Qpazo \cap \Pcal_c(\R^d)$, there exists a curve $\mu(\cdot) \in \Spazo_{[0,T]}(\mu^0)$ such that $\mu(t) \in \Qpazo$ for all times $t \in [0,T]$.
\end{thm}

\begin{proof}
The proof of this result relies on an estimate ``\`a la Gr\"onwall'' on the distance between $\Rpazo_t(\mu^0)$ defined as in \eqref{eq:Reachable} and $\Qpazo$. Consider $\mu^0 \in \Qpazo \cap \Pcal_c(\R^d)$ and fix $r>0$ such that $\mu^0 \in \Pcal(B(0,r))$. Let then $R_r \geq r > 0$ be as in Theorem \ref{thm:Filippov} and choose $R \geq R_r > 0$ in such a way that
\begin{equation}
\label{eq:Rdef}
\begin{aligned}
\dist_{\Pcal_2} \Big( \Rpazo_t(\mu^0) \, & ; \partial \B_{\Pcal_2}(\delta_0,R) \Big) \\
& \, \geq \dist_{\Pcal_2} \big( \Rpazo_t(\mu^0) \, ; \Qpazo \big) + 1.
\end{aligned}
\end{equation}
for all times $t \in [0,T]$. We then define the restricted constraints set $\Qpazo_R := \Qpazo \cap \B_{\Pcal_2}(\delta_0,R)$, which is compact since $\Qpazo$ is proper, along with the distance function
\begin{equation}
\label{eq:Gdef}
g : t \in [0,T] \mapsto \dist_{\Pcal_2} \Big( \Rpazo_t(\mu^0) \, ; \Qpazo_R \Big).
\end{equation}
It can be checked easily that $g(\cdot) \in \AC([0,T],\R_+)$, and we denote by $\Dcal_g \subset (0,T)$ the set of full $\Lcal^1$-measure where it is differentiable.


\paragraph*{Step 1 -- Distance estimate} Noticing at first that $g(0) = 0$ by construction, we claim that $g(\cdot) \equiv 0$ on $[0,T]$. Indeed otherwise, by the continuity of $g(\cdot)$, there exists some $t \in [0,T]$ and $\delta > 0$ such that $g(t) = 0$ and while $g(\tau) > 0$ for each $\tau \in (t,t+\delta)$. Let $\tau \in (t,t+\delta) \cap \Tcal \cap \Dcal_g$ with $\Tcal \subset (0,T)$ being defined as in Theorem \ref{thm:Velocity}, and observe that since $\Rpazo_{\tau}(\mu^0)$ and $\Qpazo_R$ are both compact -- the former by Theorem \ref{thm:Compactness} and the latter by construction --, it then holds that
\begin{equation*}
g(\tau) = W_2(\mu_{\tau},\nu_{\tau}),
\end{equation*}
for some $\mu_{\tau} \in \Rpazo_{\tau}(\mu^0)$ and $\nu_{\tau} \in \Qpazo_R$. Moreover by \eqref{eq:Rdef}, one necessarily has that
\begin{equation*}
\nu_{\tau} \in \Qpazo \cap \textnormal{int} \big( \B_{\Pcal_2}(\delta_0,R) \big).
\end{equation*}
Thus $T_{\Qpazo_R}(\nu_{\tau}) = T_{\Qpazo}(\nu_{\tau})$ and for each $\xi_{\tau} \in T_{\Qpazo}(\nu_{\tau})$, there exists a sequence $h_i \to 0^+$ such that
\begin{equation}
\begin{aligned}
W_2 \big( \mu_{\tau},(\Id+h_i\xi_{\tau})_{\sharp} \nu_{\tau} \big) & \geq \dist_{\Pcal_2}(\mu_{\tau} \, ; \Qpazo_R) + o_{\tau}(h_i) \\
& = W_2(\mu_{\tau},\nu_{\tau}) + o_{\tau}(h_i).
\end{aligned}
\end{equation}
Hence, by fixing an arbitrary $\gamma_{\tau} \in \Gamma_o(\mu_{\tau},\nu_{\tau})$ and applying Proposition \ref{prop:WassSuperdiff} with $(\pi^2,\pi^1)_{\sharp} \gamma_{\tau} \in \Gamma_o(\nu_{\tau},\mu_{\tau})$, it holds
\begin{equation*}
\begin{aligned}
o_{\tau}(1) & \leq \tfrac{1}{2h_i} W_2^2 \big( \mu_{\tau},(\Id+h_i\xi_{\tau})_{\sharp} \nu_{\tau} \big) - \tfrac{1}{2h_i} W_2^2(\mu_{\tau},\nu_{\tau}) \\
& \leq \INTDom{\langle \xi_{\tau}(y) , y-x \rangle}{\R^{2d}}{\gamma_{\tau}(x,y)} + h_i \NormL{\xi_{\tau}}{2}{\nu_{\tau}},
\end{aligned}
\end{equation*}
and we subsequently obtain upon letting $i \to +\infty$ that
\begin{equation}
\label{eq:ViabEst1}
\INTDom{\langle \xi_{\tau}(y) , x-y \rangle}{\R^{2d}}{\gamma_{\tau}(x,y)} \leq 0,
\end{equation}
for each $\xi_{\tau} \in T_{\Qpazo}(\nu_{\tau})$ and all $\gamma_{\tau} \in \Gamma_o(\mu_{\tau},\nu_{\tau})$.

On another note, by Theorem \ref{thm:Velocity}, there exists for every $v_{\tau} \in V(\tau,\mu_{\tau})$ a solution $\mu(\cdot)$ of \eqref{eq:CI} such that
\begin{equation*}
W_2 \big( \mu(\tau+h) , (\Id + h v_{\tau})_{\sharp} \mu_{\tau} \big) = o_{\tau}(h)
\end{equation*}
for any small $h > 0$. Whence, we can estimate the forward difference quotient of $\tfrac{1}{2}g^2(\cdot)$ at $\tau \in \Tcal \cap \Dcal_g$ as
\begin{equation}
\label{eq:Gest1}
\begin{aligned}
\tfrac{1}{2h} \big( g^2(\tau+h) - g^2(\tau)\big) & \leq \tfrac{1}{2h} W_2^2 \big( (\Id + h v_{\tau})_{\sharp} \mu_{\tau} , \nu_{\tau} \big) \\
& \hspace{0.8cm} - \tfrac{1}{2h} W_2^2(\mu_{\tau},\nu_{\tau}) + o_{\tau}(1),
\end{aligned}
\end{equation}
where we used the fact that $\nu_{\tau} \in \Qpazo_R$. Besides, by Proposition \ref{prop:WassSuperdiff}, it holds for each $\gamma_{\tau} \in \Gamma_o(\mu_{\tau},\nu_{\tau})$ that
\begin{equation}
\label{eq:ViabEst2}
\begin{aligned}
\tfrac{1}{2h} W_2^2 \big( (&\Id + h v_{\tau})_{\sharp} \mu_{\tau} , \nu_{\tau} \big) - \tfrac{1}{2h} W_2^2(\mu_{\tau},\nu_{\tau}) \\
& \leq \INTDom{\langle v_{\tau}(x) , x-y \rangle}{\R^{2d}}{\gamma_{\tau}(x,y)} + h \NormL{v_{\tau}}{2}{\mu_{\tau}}.
\end{aligned}
\end{equation}
Thus, upon merging \eqref{eq:Gest1} and \eqref{eq:ViabEst2} while letting $h \to 0^+$, we further obtain
\begin{equation}
\label{eq:Gest2}
g(\tau) \dot g(\tau) \leq \INTDom{\langle v_{\tau}(x) , x-y \rangle}{\R^{2d}}{\gamma_{\tau}(x,y)},
\end{equation}
for any $\tau \in (t,t+\delta) \cap \Tcal \cap \Dcal_g$ and each $\gamma_{\tau} \in \Gamma_o(\mu_{\tau},\nu_{\tau})$. Then, it follows by inserting crossed terms in \eqref{eq:Gest2} that
\begin{equation}
\label{eq:Gest3}
\begin{aligned}
g(\tau) \dot g(\tau) & \leq \INTDom{\langle v_{\tau}(x) - v_{\tau}(y) , x-y \rangle}{\R^{2d}}{\gamma_{\tau}(x,y)} \\
& \hspace{0.45cm} + \INTDom{\langle v_{\tau}(y) - \xi_{\tau}(y) , x-y \rangle}{\R^{2d}}{\gamma_{\tau}(x,y)} \\
& \hspace{0.45cm} + \INTDom{\langle \xi_{\tau}(y) , x-y \rangle}{\R^{2d}}{\gamma_{\tau}(x,y)} \\
& \hspace{-0.45cm} \leq l(\tau) g^2(\tau) + \INTDom{\langle v_{\tau}(y) - \xi_{\tau}(y) , x-y \rangle}{\R^{2d}}{\gamma_{\tau}(x,y)},
\end{aligned}
\end{equation}
where we used \eqref{eq:Gdef}, \eqref{eq:ViabEst1} and hypothesis \ref{hyp:CI}-$(iii)$. Recall now that by hypothesis \ref{hyp:CI}-$(iv)$ along with the definition of $\Tcal$, there exists for every $w_{\tau} \in V(\tau,\nu_{\tau})$ some other element $v_{\tau} \in V(\tau,\mu_{\tau})$ for which
\begin{equation*}
\sup_{x \in \R^d} |v_{\tau}(x) - w_{\tau}(x)| \leq L(\tau) W_2(\mu_{\tau},\nu_{\tau}).
\end{equation*}
Observing that \eqref{eq:Gest3} holds for every $v_{\tau} \in V(\tau,\mu_{\tau})$, this further implies that
\begin{equation}
\label{eq:Gest4}
\begin{aligned}
g(\tau) \dot g(\tau) & \leq \big( l(\tau) + L(\tau) \big)g^2(\tau) \\
& \hspace{0.45cm} + \INTDom{\langle w_{\tau}(y) - \xi_{\tau}(y) , x-y \rangle}{\R^{2d}}{\gamma_{\tau}(x,y)},
\end{aligned}
\end{equation}
for all $\xi_{\tau} \in T_{\Qpazo}(\nu_{\tau})$ and each $w_{\tau} \in V(\tau,\nu_{\tau})$. Noticing in turn that the right-hand side of the previous identity is both linear and continuous with respect to $\xi_{\tau} \in L^2(\R^d,\R^d;\nu_{\tau})$, one can deduce that \eqref{eq:Gest4} in fact holds for each $\xi_{\tau} \in \co T_{\Qpazo}(\nu_{\tau})$. Choosing in particular
\begin{equation*}
\xi_{\tau} = w_{\tau} \in V(\tau,\nu_{\tau}) \cap \co T_{\Qpazo}(\nu_{\tau}) \neq \emptyset,
\end{equation*}
we finally recover the differential inequality
\begin{equation*}
\dot g(\tau) \leq (l(\tau) + L(\tau)) g(\tau)
\end{equation*}
that holds for all $\tau \in (t,t+\delta) \cap \Tcal \cap \Dcal_g$. Since $g(t) = 0$, it follows from Gr\"onwall's lemma that $g(\cdot) \equiv 0$ on $(t,t+\delta)$, which contradicts our working assumption.


\paragraph*{Step 2 -- Existence of a viable curve}

As a consequence of Step 1, it holds that
\begin{equation*}
\Rpazo_t(\mu^0) \cap \Qpazo \neq \emptyset
\end{equation*}
for all times $t \in [0,T]$. For each $n \geq 1$, consider the dyadic subdivision $[0,T] = \cup_{k=0}^{2^n-1} [t_k,t_{k+1}]$ of the interval $[0,T]$, where $t_k := kT/2^n$. By using the semigroup property \eqref{eq:SemigroupReach} of the reachable sets, one can prove by a simple induction argument that there exists a solution $\mu_n(\cdot) \in \Spazo_{[0,T]}(\mu^0)$ such that
\begin{equation}
\label{eq:PointwiseInc}
\mu_n(t_k) \in \Qpazo,
\end{equation}
for all $k \in \{0,\dots,2^n\}$ and each $n \geq 1$.

By repeating this process for arbitrary integers $n \geq 1$, we can find a sequence of trajectories $(\mu_n(\cdot)) \subset \Spazo_{[0,T]}(\mu^0)$ for which \eqref{eq:PointwiseInc} holds. By the compactness result of Theorem \ref{thm:Compactness}, there exists a curve $\mu(\cdot) \in \Spazo_{[0,T]}(\mu^0)$ such that
\begin{equation*}
\sup_{t \in [0,T]} W_2(\mu_n(t),\mu(t)) ~\underset{n \to +\infty}{\longrightarrow}~ 0,
\end{equation*}
along an adequate subsequence. By construction, the limit curve $\mu(\cdot) \in \AC([0,T],\Pcal_c(\R^d))$ is such that
\begin{equation*}
\mu(t_k) \in \Qpazo
\end{equation*}
for every $k \in \{0,\dots,2^n\}$ and each $n \geq 1$, which yields the thesis by a classical density argument.
\end{proof}


\section{Application to the existence of exponentially stable trajectories}
\label{section:Lyapunov}

In this last section, we provide an application of Theorem \ref{thm:Viability} to obtain the existence of trajectories for which a Lyapunov functional $\Wpazo : \Pcal_2(\R^d) \to \R_+ \cup \{+\infty\}$ with domain $\textnormal{dom}(\Wpazo) \subset \Pcal_2(\R^d)$ decays exponentially. This result will involve a suitable class of  \textit{directional lower-derivatives}, defined by
\begin{equation*}
\D_{\uparrow} \Wpazo(\mu)(\xi) := \liminf_{\substack{h \to 0^+ \hspace{-0.03cm}, \; \mu_h \in \textnormal{dom}(\Wpazo) \\ W_2(\mu_h,(\Id + h \xi)_{\sharp} \mu) = o(h)}} \hspace{-0.2cm} \frac{\Wpazo(\mu_h) - \Wpazo(\mu)}{h}
\end{equation*}
for any $\mu  $ with $ \Wpazo(\mu) < \infty$ and each $\xi \in L^2(\R^d,\R^d;\mu)$.

\begin{Def}[Strict Lyapunov functions]
A map $\Wpazo : \Pcal_2(\R^d) \to \R_+ \cup \{+\infty\}$ is a \textit{strict Lyapunov function} for $V : \R_+ \times \Pcal_2(\R^d) \rightrightarrows C^0(\R^d,\R^d)$ if the following holds.
\begin{enumerate}
\item[$(i)$] $\Wpazo(\cdot)$ has compact sublevels in $\Pcal_2(\R^d)$.
\item[$(ii)$] For $\Lcal^1$-almost every $t \geq 0$ and all $\mu \in \textnormal{dom}(\Wpazo)$, there exists a $v \in V(t,\mu)$ for which
\begin{equation*}
\D_{\uparrow} \Wpazo(\mu)(v) \leq -\rho \Wpazo(\mu),
\end{equation*}
where $\rho > 0$ is a fixed constant.
\end{enumerate}
\end{Def}

\begin{thm}[Exponentially stable trajectories]
\label{thm:Lyapunov}
Let $V : \R_+ \times \Pcal_2(\R^d) \rightrightarrows C^0(\R^d,\R^d)$ be a set-valued map with convex images satisfying hypotheses \ref{hyp:CI} wherein $[0,T]$ is replaced by $[0,+\infty)$, and $\Wpazo : \Pcal_2(\R^d) \to \R_+ \cup \{+\infty\}$ be a strict Lyapunov function for $V(\cdot,\cdot)$. Then for each $\mu^0 \in \Pcal_c(\R^d)$, there exists a curve $\mu(\cdot) \in \Spazo_{[0,+\infty)}(\mu^0)$ such that
\begin{equation*}
\Wpazo(\mu(t)) \leq \Wpazo(\mu^0) e^{- \rho t},
\end{equation*}
for all times $t \geq 0$.
\end{thm}

\begin{proof}
We consider the extended dynamical system
\begin{equation}
\label{eq:CIExt}
\left\{
\begin{aligned}
& \partial_t \Bmu(t) \in - \Div_{(x,y)} \Big( \Vcal(t,\Bmu(t)) \Bmu(t) \Big), \\
& \Bmu(0) = \mu^0 \times \delta_{\Wpazo(\mu^0)},
\end{aligned}
\right.
\end{equation}
whose right-hand side is defined by
\begin{equation*}
\Vcal(t,\Bmu) := \Big\{ \big(v, - \rho \mathsmaller{\INTDom{y \,}{\R^{d+1}}{\Bmu(x,y)}} \big) ~\, \text{s.t.}~ v \in V(t,\Bpi^d_{\sharp}\Bmu) \Big\},
\end{equation*}
with $\Bpi^d : (x,y) \in \R^d \times \R \mapsto x \in \R^d$. Fixing an arbitrary $T>0$, it can be checked that this velocity field satisfies hypotheses \ref{hyp:CI}. We consider the constraints defined by
\begin{equation*}
\Qpazo := \Big\{ \Bmu \in \Pcal_2(\R^{d+1}) ~\, \text{s.t.}~ \Bmu = \mu \times \delta_y ~\text{and}~ y \geq \Wpazo(\mu) \Big\},
\end{equation*}
which is proper in $\Pcal_2(\R^{d+1})$ under our assumptions. By adapting existing results of non-smooth analysis following e.g. \cite[Proposition 2.17]{Badreddine2022}, it can be verified that
\begin{equation*}
\begin{aligned}
\D_{\uparrow} \Wpazo(\mu)(v) \leq -\rho \Wpazo(\mu) ~\Longleftrightarrow~ ( v,-\rho y) \in T_{\Qpazo}(\mu \times \delta_y),
\end{aligned}
\end{equation*}
for all $\mu \in \Pcal_2(\R^d)$, any $v \in V(t,\mu)$ and all $y \geq \Wpazo(\mu)$, which equivalently means that
\begin{equation*}
\Vcal(t,\Bmu) \cap T_{\Qpazo}(\Bmu) \neq \emptyset
\end{equation*}
for $\Lcal^1$-almost every $t \in [0,T]$ and each $\Bmu \in \Qpazo$. Thus by Theorem \ref{thm:Viability}, there exists a solution $\Bmu(\cdot)$ of \eqref{eq:CIExt} such that $\Bmu(t) \in \Qpazo$ for all times $t \in [0,T]$. Moreover by \eqref{eq:Semigroup} of Theorem \ref{thm:Wellposedness}, the curve $\Bmu(\cdot)$ admits the decomposition
\begin{equation*}
\Bmu(t) = \mu(t) \times \delta_{y(t)}
\end{equation*}
where $\mu(\cdot) \in \Spazo_{[0,T]}(\mu^0)$ satisfies
\begin{equation*}
\Wpazo(\mu(t)) \leq y(t) = e^{-\rho t} \Wpazo(\mu^0),
\end{equation*}
for all times $t \in [0,T]$.

We can then extend $\mu(\cdot)$ to $\R_+$ by repeating this process on intervals of the form $[nT,(n+1)T]$ while replacing $\Wpazo(\mu^0)$ by $\Wpazo(\mu(nT))$ in \eqref{eq:CIExt} for $n \geq 1$, yielding
\begin{equation*}
\Wpazo(\mu(t)) \leq \Wpazo(\mu^0) e^{- \rho t}
\end{equation*}
for all times $t \geq 0$, and thus concluding our proof.
\end{proof}


{\footnotesize
\bibliographystyle{plain}
\bibliography{../ControlWassersteinBib} 
}

\end{document}